\documentclass[a4paper,12pt]{article}

\usepackage{authblk}

\usepackage{cmap}
\usepackage{physics}
\usepackage[utf8]{inputenc}
\usepackage[english]{babel}
\usepackage{amsmath,amssymb,amsthm}
\usepackage{geometry}
\usepackage[colorlinks,linkcolor=blue]{hyperref}
\usepackage{graphicx}
\graphicspath{ {./Images/} }

\newtheorem{thm}{Theorem}[section]
\newtheorem{lem}[thm]{Lemma}
\newtheorem{defi}[thm]{Definition}

\newtheorem{col}[thm]{Corollary}
\newtheorem{rem}[thm]{Remark}

\newcounter{namedenvcounter}[section]

\renewcommand{\phi}{\varphi}
\newcommand{\nul}{{\bf0}}

\newcommand{\R}{\mathbb{R}}
\newcommand{\Z}{\mathbb{Z}}
\newcommand{\N}{\mathbb{N}}
\newcommand{\T}{\mathbb{T}}
\newcommand{\h}{\widehat}
\newcommand{\w}{\widetilde}
\newcommand{\ol}{\overline}
\newcommand{\smlm}{\sum\limits}

\renewcommand{\epsilon}{\varepsilon}

\title{Approximation by quasi-projection operators and dual wavelet frames in Sobolev spaces}

\author{
    Danila Kotov\thanks{\href{https://orcid.org/0009-0009-0362-9860}{https://orcid.org/0009-0009-0362-9860}}, 
    Aleksandr Krivoshein\thanks{\href{https://orcid.org/0000-0002-4619-474X}{https://orcid.org/0000-0002-4619-474X}} 
    \\
    {\small 
        Department of Applied Mathematics and Control Processes\\
        Saint Petersburg State University, Russia \\
        st117038@student.spbu.ru, a.krivoshein@spbu.ru
    }
}

\date{}

\begin{document}
\maketitle

\abstract{For a quasi‑projection operator $Q_j(f,\phi, \w\phi)$ formed by a
compactly supported function $\phi$ and a compactly supported distribution $\w\phi$ with dilation matrix $M$, we establish necessary and sufficient conditions under which it provides prescribed simultaneous approximation and simultaneous density orders for a function $f$ from the Sobolev space $H^s(\R^d)$ and for its derivatives. 
 The obtained criteria are then used to endow MRA‑based dual wavelet frames in a pair of dual Sobolev spaces with the desired simultaneous approximation properties.
}

Keywords:     quasi-projection operators, Sobolev spaces, approximation order,
density order, Strang–Fix conditions,  MRA-based wavelets, dual wavelet frames

MSC2020: 41A25, 41A35, 42C15, 42C40

\section{Introduction}

Wavelet systems are a fundamental tool for solving many problems in signal processing. The flexibility inherent in their construction allows one to endow them with a specific set of properties tailored to a particular application. Among such tunable properties are time–frequency localization, approximation order, smoothness, symmetry, support size and the type of the system (such as orthogonal/biorthogonal bases, frames, frame‑like systems, multi‑wavelets, wavelets on locally compact abelian groups, etc.).
This paper is devoted to studying the conditions under which wavelet systems and related quasi‑projection operators provide a prescribed approximation order in Sobolev norms. This is useful in cases where one needs to provide a good approximation not only of the function itself, but also of its derivatives up to a chosen order.

A quasi-projection operator is defined by
\begin{equation}
Q_j(f) = Q_j(f,\phi, \w\phi) =|\det M| \sum\limits_{k \in \Z^d} \langle f, \w\phi(M^j\cdot+k)\rangle \ \phi(M^j\cdot+k),
\label{QjDef}
\end{equation}
where $j\in\Z$, $M$ is a dilation matrix and
and $\phi$, $\w\phi$ are given functions.
Such operators include classic orthogonal projections in $L_2(\R^d)$, 
the Whittaker–Shannon sampling formula, and the Kantorovich–Kotelnikov type operators.
Approximation properties of the quasi‑projection operators
\( Q_j(f) \) as \( j \to +\infty \)
have been actively studied in the literature in various settings (see, for instance,~\cite{Butzer2006, Butzer2010, Butzer2005, Lebedeva,
        Jia, KKS, KS, KrS11, KrS17} and the references therein). 
        In particular, approximation by quasi‑projection operators in Sobolev spaces was also investigated.
In \cite{Kyriazis} the properties of quasi‑projection operators in Triebel–Lizorkin spaces were studied, and, in particular, it was established that for functions \( \phi \) and \( \widetilde{\phi} \) decaying sufficiently fast and bounded together with their derivatives up to order \( n \), under the conditions that $\phi$ satisfies the Strang–Fix condition of order $n$, $\widehat{\phi}(0) = 1$, $ D^{\beta} \widehat{\phi}(0) = 0$ for all $1 \le \beta < n$,
    \( \widehat{\widetilde{\phi}}(0) = 1 \),
    \( D^{\beta} \widehat{\widetilde{\phi}}(0) = 0 \) for all \( 1 \le \beta < n \),
the following inequality holds:
\begin{equation}
    |f - Q_j(f, \phi, \widetilde{\phi})|_{W_p^s} \le C \, m^{-j(n-s)} |f|_{W_p^n}, \quad j \ge 0,   
    \label{fMainEst1D}
\end{equation}
for all \( f \in W_p^n(\mathbb{R}) \), where \( 0 \le s < n \), \( s, n \in \mathbb{N} \), \( 1 \le p \le \infty \), $M=m > 1$ and $|\cdot|_{W_p^n}$ is
the Sobolev seminorm, i.e. $|f|_{W_p^n} = \|D^n f\|_{L_p}$. 
In~\cite{Jia} approximation by quasi-projection operators in Besov spaces was studied and, in particular, it was established that
inequality~(\ref{fMainEst1D}) holds for all $f\in W_p^n(\R)$ under the conditions that $\phi \in W_p^s(\R)$, $\w\phi \in L_q(\R)$, $\frac 1p + \frac 1q = 1$, and $Q_0(w, \phi, \w\phi) = w$ for any polynomial $w$ of degree less than $n$.

The aim of this paper is twofold. First, we obtain an estimate of the form~(\ref{fMainEst1D}) in the case, where $p=2$, $\w\phi$ is a compactly supported distribution and $\phi$ is a compactly supported function which belong to a pair of dual Sobolev spaces $H^{-s}(\R^d)$ and $H^{s}(\R^d)$, respectively. Second, we prove the necessity of the conditions that guarantee such an estimate.
The main contributions with respect to the aforementioned results are the proof of necessity and the fact that
$\w\phi$ can be taken as a distribution.
This setting is particularly useful for applying the obtained results to study
the approximation properties of MRA-based wavelets.

The paper is organized as follows. 
In Section 2 notation and auxiliary results are introduced.
In Section~3, we establish necessary and sufficient conditions under which the quasi‑projection operator $Q_j$ attains a prescribed simultaneous approximation order.
Conditions for a stronger notion of simultaneous density order are also discussed. 
In Section~4 the obtained results are used to show how MRA‑based dual wavelet frames in a pair of dual Sobolev spaces can be endowed with the desired simultaneous approximation properties.

\section{Notation and auxiliary results}

For vectors $x, y\in\R^d$ their inner product
    is denoted by $(x, y)$;
    the standard norm of a vector $x\in\R^d$ is denoted by
    $|x| = \sqrt {(x, x)}$,
    $B_{\delta}(x)$ denotes the open ball of radius $\delta$ centered at $x$,
    $\nul$ is the origin in $\R^d$,
    $\N_0 = \N \cup \{0\},$
    $\T^d = [-\frac12,\frac12)^d$.
    For a function $f$ defined on $\R^d$
    we set $f_{jk}(x) = m^{j/2} f(M^j x + k)$, where $j\in\Z$, 
    $k\in\Z^d$, $x\in \R^d$ and $M$ here and below denotes a dilation matrix, 
    i.e.
    $M\in \Z^{d\times d}$ and all its eigenvalues are greater than 1 
    in modulus, $m = |\det M|.$
    
    The spectrum of 
	matrix $M^{-1}$ (i.e. the set of its eigenvalues) is located inside the ball centered at the origin with
    radius $\rho(M^{-1})$,
	where $\rho(M)=\lim\limits_{j\to+\infty}\|M^{j}\|^{1/j}$
	is the \textit{spectral radius} of matrix $M$. Since $\rho(M^{-1})$ is
    also equal to the maximum eigenvalue of $M^{-1}$ in modulus, 
    then $\rho(M^{-1})<1.$ 
    Also, for every $\theta >\rho(M^{-1})$ there exists $C_{M,\theta}>0$ such that
    $\|M^{-j}\|\le {C_{M,\theta}}\,\theta^j$  for all $j\ge0$.
    	In particular, we can take $\theta < 1$, 
        which yields that $\lim\limits_{j\to+\infty}\|M^{-j}\|=0.$
        Moreover, we always can take $N > 0$ such that $\theta^N m < 1$ and, hence,
\begin{equation}
        m^j \|M^{-j}\|^N \le m^j \ {C^N_{M,\theta}} \ \theta^{jN} \to 0
        \quad \text{as} \ j \to +\infty.
\label{fMjEstimation2}
\end{equation}
A dilation matrix $M$ is called \textit{isotropic} if all its eigenvalues $\lambda_1,\dots,\lambda_d$ are equal 
    in modulus, i.e. $\rho = \rho(M) = |\lambda_i| = m^{1/d}$, $i=1,\dots,d$.
    	Let $M$ be an isotropic matrix with spectral radius
	$\rho=\rho(M)$. By~\cite[Lemma 6.1]{Jia1998ApproxProp} there exist two positive constants
	$C_1$ and $C_2$ such that
	\begin{equation}
	C_1 \rho^j 
	\le
	\|M^j\|
	\le
	C_2 \rho^j 
    \label{fIsotropicM}
	\end{equation}
	for all $j \in \mathbb{Z}$. Note that $M^{-1}$ is also isotropic and 
    $\rho(M^{-1}) = \frac{1}{\rho(M)} = m^{-1/d}$.

 The Fourier coefficients of $f\in L_1(\T^d)$ are 
    $c_k(f) = \int\limits_{\T^d} f(x)e^{-2\pi i (x, k)}\,dx$, $k\in\Z^d$.
    For a sequence $a=\{a(k)\}_{k \in \Z^d} \in \ell_2(\Z^d)$ its
    symbol is a 1-periodic function in $L_2(\T^d)$ defined by
    $\h a(\xi) = \sum\limits_{k \in \Z^d} a(k) e^{2\pi i k \xi}$, 
    $\xi \in \T^d$.
For a function $f\in L_1(\R^d)$ its 
    Fourier transform 
    is defined by $\widehat
    f(\xi)=\int\limits_{\R^d} f(x)e^{-2\pi i (x, \xi)}\,dx,$ $\xi\in{\mathbb R}^d$.
    By standard arguments this definition extends to functions from $L_2(\R^d)$ and to tempered distributions.
    ${\mathcal S}$ denotes the Schwartz class of functions.
    The set of finitely supported sequences is denoted by $\ell_0(\mathbb{Z}^d)$.

Let $s \in \R$. The Sobolev space 
    $H^s(\R^d)$ consists of tempered
    distributions $f$ such that
    \[
    \|f\|^2_{H^s} := \int\limits_{\R^d} |\h f (\xi)|^2 (1 + |\xi|^2)^s d\xi < +\infty.
    \]
The inner product in $H^s(\R^d)$ can be defined by
    \[
    \langle f, g \rangle_{H^s} =
    \int\limits_{\R^d} \h f(\xi) \overline{\h{g}(\xi)}
    (1 + |\xi|^2)^s d \xi
    \]
and $H^s(\R^d)$ is a Hilbert space.
It is known that spaces $H^s(\R^d)$ and $H^{-s}(\R^d)$ 
    are dual spaces (see, e.g.,~\cite{Han}) and we can define  a bilinear form
    \[
    \langle f, g \rangle := 
    \int\limits_{\R^d} \h f (\xi) \overline{\h g(\xi)} d \xi
    \] 
for $f\in H^s(\R^d)$ and $g \in H^{-s}(\R^d)$. It is clear that 
    $\langle f, g \rangle \le  \|f\|^2_{H^s} \|g\|^2_{H^{-s}}.$ 
    Also, it is convenient to use the following notation
for $f\in H^s(\R^d)$
$$
|f|_{H^s}^2 := \int_{\R^d} |\xi|^{2s} |\h f(\xi)|^2 \dd \xi.
$$
Note that $(\|f\|_{L_2}^2 + |f|_{H^s}^2)^{1/2}$ is also a norm for function $f \in H^s(\R^d)$, which is equivalent to $\|f\|^2_{H^s}.$

\begin{defi}
    Let $s\in \R$.
    For two measurable functions 
    $f, g :\R^d \rightarrow \mathbb C$ 
    we define their $s$-bracket product as the 
    1-periodic function
    \[
    [f, g]_s(\xi) := \sum\limits_{k \in \Z^d} f(\xi + k)
    \ol{g(\xi + k)}(1 + |\xi + k|^2)^s,
    \]
    if the series makes sense for a.e. $\xi \in \R^d$.

    For convenience, we will also write 
    $[f, g] := [f,g]_0$, $[f, g]^{\neq\nul} := [f, g] - f \overline{g}$ and
    $[f, g]_s^{\neq\nul} := [f, g]_s - f \overline{g} (1+|\cdot|^2)^s$.
\end{defi}

\begin{rem}
    \label{rem1}
    By the monotone convergence theorem and Lebesgue's dominated convergence theorem it is straightforward to check that for 
    $f\in H^s(\R^d)$ the $s$-bracket 
    product $[\h f, \h f]_s$ is in $L_1(\T^d)$
    and 
    $c_k([\h f, \h f]_s) = \langle f, f_{0k} \rangle_{H^s}$ for $k\in\Z^d.$
    Hence, if function $f\in H^s(\R^d)$ has compact support,
    then $[\hat{f}, \hat{f}]_s$ is a trigonometric polynomial
    and $[\hat{f}, \hat{f}]_s \in L_\infty(\R^d)$.
    
    Also, by similar arguments for $f\in H^s(\R^d)$ and $g \in H^{-s}(\R^d)$
    their $0$-bracket product $[\h f, \h g]$ is in $L_1(\T^d)$
    and
    $c_k([\h f, \h g]) = \langle f, g_{0k} \rangle$.
\end{rem}

Several statements below establish some useful inequalities and basic properties of quasi-projection operators in $H^s(\R^d)$.

\begin{lem} \emph{\cite[Theorem 2.3]{Han}}
	Let $s\in \R,$
		$\w \phi \in H^{-s}(\R^d)$ such that $[\h {\w\phi}, \h{\w \phi}]_{-s} \in L_{\infty}(\R^d),$
        $f\in H^s(\R^d)$. 
        Then $[\h f, \h{\w\phi}] \in L_2(\T^d),$ 
		$[\h f, \h{\w \phi}](\xi) = \sum\limits_{k \in \Z^d}\langle f, \w\phi_{0k} \rangle e^{2\pi i (k,\xi)} $
	and
		\[
        \norm{[\h f, \h{\w \phi}]}^2_{L_2(\T^d)} = \sum\limits_{k\in \Z^d} |\langle f, \w\phi_{0k} \rangle|^2 \le
        \norm{[\h{\w\phi}, \h{\w\phi}]_{-s}}_{L_{\infty}} \norm{f}^2_{H^s}.
		\]
    \label{lemCoef}
\end{lem}

\begin{lem}
	Let $s\in \R,$
		$\phi \in H^{s}(\R^d)$ such that $[\h {\phi}, 
        \h{\phi}]_{s} \in L_{\infty}(\R^d),$ 
        $a=\{a(k)\}_{k \in \Z^d} \in \ell_2(\Z^d)$. 
    Then the series $\sum\limits_{k\in\Z^d}a(k)\phi_{0k}$ 
    converges unconditionally in $H^s(\R^d)$ and 
		\[
        \norm{\sum_{k \in \Z^d}a(k)\phi_{0k}}^2_{H^s} \le
        \norm{[\h\phi,\h\phi]_s}_{L_\infty}
		\norm{\h a}^2_{L_2(\T^d)}.
        \]
\end{lem}

\textbf{Proof.} For any finitely supported sequence 
    $a=\{a(k)\}_{k \in \Z^d}$ we get
		\begin{align*}
			&\norm{\sum\limits_{k \in \Z^d}a(k)\phi_{0k}}^2_{H^s} = \int\limits_{\R^d}(1 + |\xi|^2)^s|\h a(\xi)|^2|\h\phi(\xi)|^2\dd\xi 
			\\
			& \hspace{1cm} =\int\limits_{\T^d}|\h a(\xi)|^2
            \sum\limits_{k \in \Z^d}
            (1 + |\xi + k|^2)^s|\h \phi(\xi + k)|^2\dd\xi \le
            \norm{[\h\phi,\h\phi]_s}_{L_\infty}
			\norm{\h a}^2_{L_2(\T^d)}.
		\end{align*} 
	Unconditional convergence of the series $\sum\limits_{k\in\Z^d}a(k)\phi_{0k}$ follows from the well‑known criterion, which requires the smallness of the norms of finite sums over indices which are far from the origin. This condition holds due to the established inequality. \qed

		\begin{thm}\label{thm1} 
	Let $s\in \R,$
			$\phi \in H^s(\R^d)$, $\w\phi \in H^{-s}(\R^d)$ such that $[\h\phi,\h\phi]_s$ and $[\h{\w\phi}, \h{\w\phi}]_{-s}$ 
            are in $L_\infty(\R^d)$, $f\in H^s(\R^d)$. 
            Then the series 
            $\sum\limits_{k \in \Z^d}\langle f,\w\phi_{0k} \rangle \phi_{0k}$ converges unconditionally in $H^s(\R^d)$ and
			\[
            \norm{\sum\limits_{k \in \Z^d} \langle f, \w\phi_{0k}\rangle\phi_{0k} }_{H^s}^2 \le 
            \norm{[\h{\w\phi},\h{\w\phi}]_{-s}}_{L_\infty} 
            \norm{[\h\phi,\h\phi]_s}_{L_\infty}\norm{f}^2_{H^s}.
			\]
		\end{thm}
		This theorem follows directly from the preceding lemmas.

\begin{lem} 
	Let $s\in \R,$
		$\w \phi \in H^{-s}(\R^d)$ such that $[\h {\w\phi}, \h{\w \phi}]_{-s} \in L_{\infty}(\R^d),$
        $f\in H^s(\R^d)$. Then
		\[
			\int\limits_{\T^d} 
			\left| [\h f, \h{\w \phi}]^{\neq\nul}(\xi)\right|^2\dd\xi 
			\le C_{\w\phi}\int\limits_{\R^d\setminus\T^d} |\xi|^{2s} |\h f(\xi)|^2 \dd \xi.
		\]        
\label{lem3}
    \end{lem}
    
\textbf{Proof.} 
	By the Cauchy inequality we get
	\begin{align*}
		\int\limits_{\T^d} \left|[\h f, \h{\w \phi}]^{\neq\nul}(\xi)\right|^2 & \dd\xi \le
         \int\limits_{\T}\left|\sum\limits_{k \ne \nul} |\h f(\xi + k)|(1 + |\xi + k|^{2})^{\frac{s}{2}} \frac{|\h{\w \phi}(\xi+k)|}{(1 + |\xi + k|^{2})^{\frac{s}{2}}} \right|^2 
        \\
		& \le \int\limits_{\T^d}
        \sum\limits_{k \ne \nul} |\h f(\xi + k)|^2(1 + |\xi + k|^{2})^s
        \sum\limits_{l \ne \nul} \frac{|\h{\w \phi}(\xi+l)|^2}{(1 + |\xi + l|^{2})^{s}} \dd\xi 
        \\
		& \le \left\|[\h{\w\phi}, \h{\w\phi}]_{-s}\right\|_{L_\infty}
        \int\limits_{\R^d\setminus\T^d}(1 + |\xi|^{2})^s|\h f(\xi)|^2\dd \xi \le C_{\w\phi}\int\limits_{\R^d\setminus\T^d} |\xi|^{2s}|\h f(\xi)|^2\dd\xi. \qed
	\end{align*}

\begin{lem}
        Let $s \ge 0,$
		function $\w \phi \in H^{-s}(\R^d)$ has compact support, 
        $f \in \mathcal{S}$, $j \in \mathbb{Z}$. Then 
        $[\h f(M^{*j} \cdot), \h {\w \phi}] \in L_2(\T^d)$. 
        Moreover, 
        for any $N \ge s$ there exists a constant 
        $C_{N,\w\phi,f} > 0$ such that for any $\xi \in \T^d$ and any $j \ge 0$
        \[
        \left| [\h f(M^{*j} \cdot), \h {\w \phi}]^{\neq \nul}(\xi) \right| 
        \le C_{N,\w\phi,f} \, \| M^{*-j} \|^{N}.
        \]
        Also, there exists a constant $C_{\w\phi,f}$ such that
        $\left|[\h f(M^{*j} \cdot), \h {\w \phi}]\right| \le C_{\w\phi,f}$
        on $\T^d$ for any  $j \ge 0$.
        \label{lemGjMain}
        \end{lem}

\textbf{Proof.} Note that 1-periodic function 
    $[\h f(M^{*j} \cdot), \h {\w \phi}]$ is 
    well-defined and in $L_2(\T^d)$ for any $j\in\Z$ by Lemma~\ref{lemCoef},
    since $[\h {\w\phi}, \h{\w \phi}]_{-s} \in L_{\infty}(\R^d)$ by Remark~\ref{rem1}. Suppose $j \ge 0$. Then by the Cauchy inequality
    \begin{align*}
    \left| [\h f(M^{*j} \cdot), \h {\w \phi}]^{\neq \nul}(\xi) \right|
    &\le
    \sum_{l \neq \nul} (1+|\xi + l|^2)^{\frac s2} |\h f\bigl( M^{*j} (\xi + l) \bigr)|  
    \frac{|\h{\w \phi}(\xi + l)|}{(1+|\xi + l|^2)^{\frac s2}} 
    \\
    \le &\left\|[\h{\w\phi}, \h{\w\phi}]_{-s}\right\|_{L_\infty}
    \sum_{l \neq \nul} (1+|\xi + l|^2)^{N} |\h f\bigl( M^{*j} (\xi + l) \bigr)|^2 
    \\
    \le & \left\|[\h{\w\phi}, \h{\w\phi}]_{-s}\right\|_{L_\infty}
    5^N \|M^{*-j}\|^{N}
    \sum_{l \neq \nul} |M^{*j} (\xi + l)|^{N} |\h f\bigl( M^{*j} (\xi + l) \bigr)|^2,    
    \end{align*}
    since $1+|\xi + l|^2 \le 5 |\xi + l|^2$ for $\xi \in \T^d$ and $l\in\Z^d$, $l\neq\nul$. 
    It remains to note that the last series is uniformly bounded, since $f \in {\mathcal S}.$ The boundedness of 
    $[\h f(M^{*j} \cdot), \h {\w \phi}]$ follows from  the above inequality
    and from the boundedness of $\h f(M^{*j}\cdot)\h{\w\phi}$ on $\T^d$. 
    $\qed$ 
        
\section{Approximation by quasi-projection operators}

    Suppose $s \ge 0,$ functions $\phi \in H^s(\R^d)$ and 
        $\w \phi \in H^{-s}(\R^d)$ have compact supports.
        The main result of this section is to show that high order of approximation for the quasi‑projection operator $Q_j(f)$ on functions from $H^s({\mathbb R}^d)$
		can be ensured under the following conditions
		\begin{align} \label{cond6}
			\left| 1 - \ol{\h{\w \phi} (\xi)}\h \phi (\xi)  \right| \le 
            C_{\phi,\w\phi}|\xi|^n, \quad 
			[\h\phi,\h\phi]_s^{\neq\nul}(\xi) = \sum\limits_{k \ne \nul}(1 + |\xi + k|^{2})^s\left|\h{\phi}(\xi+k)\right|^2 \le C_{\phi}|\xi|^{2n}
		\end{align}
		for all $\xi \in \T^d$ with some $n > 0$.

	\begin{thm}
    \label{thmMain} 
    Let $s \ge 0$, functions
		$\phi \in H^s(\R^d)$ and $\w \phi \in H^{-s}(\R^d)$ 
        have compact supports.
		\newline
		\phantom{}\hspace{0.5cm}1. Then for $f \in H^s(\R^d)$ the series 
        $\sum\limits_{k\in \Z^d} \langle f, \w\phi_{0k} \rangle \phi_{0k}$ 
        converges unconditionally in $H^{s}(\mathbb{R}^d)$.
		\newline
		\phantom{}\hspace{0.5cm}2. If, additionally, for some $n \ge 0$ conditions (\ref{cond6}) are satisfied, then the following inequality holds
		\begin{equation}
				\norm{ f - \sum\limits_{k \in \Z^d} \langle f, \w{\phi}_{0k} \rangle \phi_{0k} }_{{ H^{s}}}^2 
				\le C\left(\norm{|\cdot|^n \h{f}}_{L_2(\mathbb{T}^d)}^2  
				+ \norm{|\cdot|^{s} \h{f}}_{L_2(\mathbb{R}^d\setminus\mathbb{T}^d)}^2\right),
				\label{fThm1}
		\end{equation}   
        where $C$ depends only  on $\phi$, $\w\phi$ and $s$.
	\end{thm}
		
    \textbf{Proof.} 
    Item 1 holds by Theorem~\ref{thm1} 
        taking into account Remark~\ref{rem1}. Let us prove Item 2.
        By Lemma~\ref{lemCoef}, Lemma~\ref{lem3} and conditions~(\ref{cond6}) 
        the following chain of inequalities hold
        \begin{align*}
		&\norm{ f - \sum\limits_{k \in \mathbb{Z}} \langle f, \w{\phi}_{0k} \rangle \phi_{0k} }_{{ H^{s}}}^2 
		= \int\limits_{\T^d \cup (\R^d \setminus \T^d)}(1+|\xi|^2)^s 
		\left|\h f(\xi) - [\h f, \h{\w\phi}](\xi)\h{\phi}(\xi) \right|^2\dd\xi 
		\\
		\le  2\int\limits_{\T^d} & (1 + |\xi|^2)^s \left(
        \left| 1 - \ol{\h{\w \phi} (\xi)}\h \phi(\xi)\right|^2 \left|\h f(\xi) 
        \right|^2 
		+  \left| [\h f, \h{\w \phi}]^{\neq\nul}(\xi) \h \phi(\xi) \right|^2 \right) \dd \xi
        \\
        & + 2\int\limits_{\R^d \setminus \T^d} (1 + |\xi|^2)^s |\h f(\xi)|^2 \dd \xi
		+ 2\int\limits_{\R^d\setminus\T^d} (1 + |\xi|^2)^s \left| [\h f, \h{\w \phi}](\xi)\h \phi(\xi) \right|^2 \dd \xi 
        \\
        \le 2^{s+1} &  C_{\phi,\w\phi} \int\limits_{\T^d} |\xi|^{2n} |\h f(\xi)|^2\dd \xi + 
		2 (2^s C_{\w\phi} \|\h{\phi}\|_{L_{\infty}(\T^d)}^2 + 5^s) \int\limits_{\R^d\setminus\T^d} |\xi|^{2s}|\h f(\xi)|^2\dd\xi + 2 I_4,
		\end{align*}
		where
		$ I_4 := \int\limits_{\R^d\setminus\T^d} (1 + |\cdot|^2)^s \left| [\h f, \h{\w \phi}]\h \phi \right|^2.$
        To get the required inequality~(\ref{fThm1}) it remains to estimate 
		integral $I_4$ 
        \begin{align*}
		I_4 
		& \leq 2  \sum\limits_{k\ne \nul} \int\limits_{\T^d} (1 + |\xi+k|^2)^s
        \left(\left| \h f(\xi)\ol{\h{\w \phi}(\xi)}\h \phi(\xi + k) \right|^2 
		+ \left|[\h f,\h{\w \phi}]^{\neq\nul}(\xi)\h\phi(\xi + k)\right|^2 
        \right)\dd\xi 
        \\
        & \le 2 \|\h{\w{\phi}}\|_{L_{\infty}(\T^d)}^2 \int\limits_{\T^d} |\h f(\xi)|^2 [\h\phi,\h\phi]_s^{\neq\nul}(\xi) \dd\xi + 2 \norm{[\h\phi,\h\phi]_s}_{L_\infty}
        \int\limits_{\T^d}  \left|[\h f,\h{\w \phi}]^{\neq\nul}(\xi)\right|^2 \dd\xi
		\\
		& \le 2 C_{\phi}  \|\h{\w{\phi}}\|_{L_{\infty}(\T^d)}^2 \int\limits_{\T^d} |\xi|^{2n}|\h f(\xi)|^2\dd\xi + 2 C_{\w\phi} \norm{[\h\phi,\h\phi]_s}_{L_\infty} \int\limits_{\R^d\setminus\T^d}|\xi|^{2s}|\h f(\xi)|^2  \dd\xi,
		\end{align*}
		where the last inequality follows from Lemma~\ref{lem3} and the second inequality in~(\ref{cond6}).
		\qed

\begin{thm}
    Let $s \ge 0$, functions
		$\phi \in H^s(\R^d)$ and $\w \phi \in H^{-s}(\R^d)$ 
        have compact supports. Then the quasi-projection operator 
        $Q_j(f)$ for $j\in\Z$ is well-defined by~(\ref{QjDef}) for functions $f \in H^s(\R^d)$.
        
        If, moreover, for some $n \ge s$ conditions~(\ref{cond6}) are satisfied,
        then for any $f \in H^n({\mathbb R^d})$ and $0 \le r \le s$ the following inequality holds
            \begin{equation}
            \left| f - \smlm_{k\in \Z^d} \langle f, \w \phi_{jk} \rangle \phi_{jk} \right|_{H^r}^2 \leq C \norm{M^{*j}}^{2r}\norm{M^{*-j}}^{2n} |f|_{H^n}^2,
            \label{fMainAppr}
            \end{equation}
        where $C$ does not depend on $j$ and $f$.
    \label{theoMainAppr}
    \end{thm}
    
\textbf{Proof.} Note that $\langle f, \w \phi_{jk} \rangle = m^{-j/2}\langle f(M^{-j}\cdot), \w\phi_{0k}\rangle$. Then by Item 1 in Theorem~\ref{thmMain}  $Q_j(f)$ is well-defined for $f \in H^s(\R^d)$ and any $j\in\Z$.

Next, using the properties of the Fourier transform and by changing variables we get
    \begin{align*}
    &\left|f - \smlm_{k\in \Z^d} 
        \langle f, \w \phi_{jk} \rangle \phi_{jk}\right|_{H^r}^2 
        = \int\limits_{\R^d} |\xi|^{2r}
    \left|\h f(\xi) - \smlm_{k \in \Z^d}m^{-j}\langle f(M^{-j}\cdot), \w\phi_{0k}\rangle \h{\phi_{0k}}(M^{*-j}\xi) \right|^2\dd \xi 
    \\
    & \hspace{1cm} =  m^{j}\int\limits_{\R^d} |M^{*j}\xi|^{2r} 
    \left| m^{-j} \h{f(M^{-j}\cdot)} (\xi) - \smlm_{k \in \Z^d} m^{-j}\langle f(M^{-j}\cdot), \w\phi_{0k}\rangle \h{\phi_{0k}}(\xi) \right|^2 \dd \xi
    \\
     & \hspace{1cm} \le  m^{-j}\norm{M^{*j}}^{2r} \left|
     f(M^{-j}\cdot)-\smlm_{k \in \Z^d}\langle f(M^{-j}\cdot), \w\phi_{0k} \rangle \phi_{0k} \right|^2_{H^r}.
    \end{align*}
    Hence, applying Item 2 of Theorem \ref{thmMain} (with $s=r$) and taking into account that $n \ge r$, we get
    \begin{align*}
    & \left|f - \smlm_{k\in \Z^d} 
       \langle f, \w \phi_{jk} \rangle \phi_{jk}\right|_{H^r}^2  \le
    C m^{-j}\norm{M^{*j}}^{2r} \int\limits_{\R^d}|\xi|^{2n}
    \left| \h{f(M^{-j} \cdot)}(\xi) \right|^2 \dd\xi =
    \\
    &  \hspace{1cm} =  C \norm{M^{*j}}^{2r}  \int\limits_{\R^d} |M^{*-j}\xi|^{2n}\left| \h{f}(\xi) \right|^2 \dd\xi \le 
    C \norm{M^{*j}}^{2r}\norm{M^{*-j}}^{2n} |f|_{H^n}^2.
    \qed
    \end{align*}

    It $M$ is an isotropic dilation matrix, then under the assumptions of Theorem~\ref{theoMainAppr} inequality~(\ref{fMainAppr}) can be written as
    \[
    \left| f - Q_j(f) \right|_{H^r} \leq C \ C_2^{n+r} \rho^{-j(n-r)} |f|_{H^n},
    \]
    where $\rho = \rho(M)$ and $C_2$ is from inequality~(\ref{fIsotropicM}). 
    Since $0 \le r \le s$, from this inequality we can deduce that 
    $$
    \left\| f - Q_j(f) \right\|_{H^s} \le C \rho^{-j(n-s)} |f|_{H^n}.
    $$
    If $s\in \N$, then the quasi-projection operator $Q_j(f)$ provides approximation order not lower than $n-s$ 
    for $f\in H^n({\mathbb R^d})$ and its derivatives up to order $s$. 
    Somewhat better description of the above considerations can be expressed in terms of
    \textit{simultaneous approximation order}. This notion was introduced in~\cite{Zhao}
    for studying the approximation properties of
    shift-invariant spaces in the Sobolev spaces. 
    In the context of approximation by the quasi-projection operators this notion 
    can be paraphrased.
    
    \begin{defi} 
    Let $s \in \N_0$ and $n > s$,
    $\phi \in H^s(\R^d)$, $\w \phi \in H^{-s}(\R^d)$, $n > s$ and
    the quasi-projection operator $Q_j(f)$ is well-defined by~(\ref{QjDef}) for functions $f \in H^s(\R^d)$. We say that $Q_j$ provides 
    simultaneous approximation order $(s,n)$ in $H^s(\R^d)$, if
    for all functions $f \in H^n(\R^d)$
    \begin{equation}
        \sum_{r=0}^{s} \norm{M^{*j}}^{-r}
        \left| f - \sum\limits_{k\in\Z^d}\langle f, \w\phi_{jk}\rangle\phi_{jk} \right|_{H^{r}} \le
        C \norm{M^{*-j}}^{n} |f|_{H^n},
        \label{fApprox}  
    \end{equation}
    where $C$ does not depend on $j$ and $f$.
    \end{defi}


Theorem~\ref{theoMainAppr} states that conditions~(\ref{cond6}) 
    are sufficient for simultaneous approximation order $(s, n)$ of $Q_j$ in $H^s(\R^d)$.
    Moreover, these conditions are also necessary.

\begin{thm}
    Let $s \in \N_0$,
    functions $\phi \in H^s(\R^d)$, $\w \phi \in H^{-s}(\R^d)$ have compact supports.
        Then the quasi-projection operator $Q_j$ provides simultaneous approximation order $(s,n)$ in $H^s(\R)$ for some $n>s$
        if and only if conditions~(\ref{cond6}) are satisfied.
    \label{fQjApprox}
\end{thm}
\textbf{Proof.} 
Sufficiency follows from Theorem~\ref{theoMainAppr}. Let us prove the necessity. 
Inequality~(\ref{fApprox})
with $j=0$ and $r=0$ yields that 
\[
\left|f - \sum\limits_{k \in \Z^d} \langle f, \w 	\phi_{0k} \rangle \phi_{0k} \right|_{H^0(\T^d)}^2  
\le 
\left|f - \sum\limits_{k \in \Z^d} \langle f, \w 	\phi_{0k} \rangle \phi_{0k} \right|_{H^0}^2 
        \leq C |f|^2_{H^n}. 
\]
Hence, for a function $f\in H^n(\R^d)$ with $\texttt{supp}~\h f \subset \T^d$ the following inequality holds
\begin{align*}
   \left|f - \sum\limits_{k \in \Z^d} \langle f, \w 	\phi_{0k} \rangle \phi_{0k} \right|_{H^0(\T^d)}^2  & = \int\limits_{\T^d} 
   \left| \h f(\xi) -  [\h f, \h{\w\phi}](\xi) \h \phi(\xi)\right|^2 \dd \xi  
   \\
   &=\int\limits_{\T^d} 
   |\h{f}(\xi)|^{2}
   \left|1 - \h \phi(\xi) \overline{\h{\w \phi}(\xi)}\right|^2\dd \xi 
   \leq C \int\limits_{\T^d} |\xi|^{2n} |\h{f}(\xi)| ^2 \dd\xi.
\end{align*}
Therefore,
    \begin{align}
        0 \leq \int\limits_{\T^d} 
        |\h f(\xi)|^2 \left(C |\xi|^{2n} - |1 - \h \phi(\xi)\overline{\h{\w \phi}(\xi)}|^2 \right) \dd\xi,
        \label{fNecessity1}
    \end{align}
which implies that $|1 - \h \phi(\xi) \overline{\h{\w \phi}(\xi)}|^2 \leq C |\xi|^{2n}$ for $\xi \in \T^d$. Indeed, let us assume the opposite. 
Suppose there exists a point $\xi_0\in \T^d$ such that
$C |\xi_0|^{2n} - |1 - \h \phi(\xi_0) \overline{\h{\w\phi}(\xi_0)}|^2 < 0$. By the continuity of $\h\phi$ and $\h{\w\phi}$ the last inequality holds in some neighborhood $U \subset \T^d$ of 
$\xi_0$. If $f$ is in $H^n(\R^d)$ with $\texttt{supp}~\h f \subset U$, then the integral in~(\ref{fNecessity1}) is negative, which gives a contradiction.

Similarly, we can prove the second inequality in~(\ref{cond6}). For a function $f\in H^n(\R^d)$ with $\texttt{supp}~\h f \subset \T^d$ the following inequality holds
    \begin{align*}
 			\left| f - \sum\limits_{k \in \Z^d}
            \langle f, \w\phi_{0k} \rangle\phi_{0k} \right|_{H^s}^2 
            & \geq 
            \left|\sum\limits_{k \in \Z^d}
            \langle f, \w\phi_{0k} \rangle\phi_{0k}
            \right|_{H^s(\R^d\setminus\T^d)}^2 
            \\
            & = \int\limits_{\T^d} \sum\limits_{k \ne \nul}  |\xi + k|^{2s} \left|
            \h\phi(\xi + k) [\h f, \h{\w\phi}](\xi) \right|^2\dd\xi
            \\
 			&\geq C' \int\limits_{\T^d} \sum\limits_{k \ne \nul} (1 + |\xi + k|^2)^s \left|
            \h\phi(\xi + k) [\h f, \h{\w\phi}](\xi) \right|^2\dd\xi
 			\\
            &= C'\int\limits_{\T^d} [\h\phi, \h\phi]^{\neq\nul}_s(\xi) 
            \left|\h f(\xi) \h{\w \phi}(\xi)\right|^2 \dd\xi.
 	\end{align*}
Together with inequality~(\ref{fApprox})
for $j=0$ this leads to
 		\[
 			C'\int\limits_{\T^d} [\h\phi, \h\phi]^{\neq\nul}_s(\xi) 
            \left|\h f(\xi) \h{\w \phi}(\xi)\right|^2 \dd\xi  
 			\leq C \int\limits_{\T^d}  |\xi|^{2n} |\h f(\xi)|^2\dd\xi
            \]
or, equivalently,       
 			\[
            0 \leq \int\limits_{\T^d} |\h f(\xi)|^2 \left(
            C |\xi|^{2n} - C'|\h{\w\phi}(\xi)|^2[\h \phi, \h \phi]^{\neq\nul}_s(\xi) \right)
            \dd\xi.
 		\]
Again, by contradiction we can show that the last inequality implies 
    \[
 			\sum\limits_{k \ne \nul}(1 + |\xi+k|^2)^s |\h\phi(\xi + k)|^2 |\h{\w\phi}(\xi )|^2 \leq \w C |\xi|^{2n}\quad \forall \xi \in \T^d.
 		\]
At the same time, the first inequality in~(\ref{cond6}) yields that $\h{\w\phi}(0) \neq 0$. And, hence, in some small ball $B_{\delta}(\nul)$ with a radius $\delta>0$ 
the function $\h{\w\phi}$ is bounded away from zero. Thus, for all $\xi \in B_{\delta}(\nul)$
\[
    \sum\limits_{k \ne \nul}(1 + |\xi+k|^2)^s |\h\phi(\xi + k)|^2 |\h{\w\phi}(\xi )|^2 \geq C'' \sum\limits_{k \ne \nul}(1 + |\xi+k|^2)^s |\h\phi(\xi + k)|^2
\]
and, therefore,
\[
    \frac{\w C}{C''}|\xi|^{2n} \geq \sum\limits_{k \ne \nul}(1 + |\xi+k|^2)^s |\h\phi(\xi + k)|^2\quad \forall \xi \in B_{\delta}(\nul).
\]
For $\xi \in \T^d \setminus B_{\delta}(\nul)$ we always can provide that
\[
\sum\limits_{k \ne \nul}(1 + |\xi+k|^2)^s |\h\phi(\xi + k)|^2 |\h{\w\phi}(\xi )|^2 \le C_{\phi,\w\phi} \le C_{\phi,\w\phi} \frac{|\xi|^n}{\delta^n}. \qed
\]

Closely related to simultaneous approximation order is the notion of 
simultaneous density order, which was introduced in~\cite{SanAntolin2025} (see also~\cite{SanAntolin2026}) in the context
of shift-invariant spaces in the Sobolev spaces. 
Motivated by this notion and by~\cite{SanAntolinL2}, it can be reformulated
in order to characterize quasi-projection operators.

\begin{defi}
    Let $s \in \N_0$,
    $\phi \in H^s(\R^d)$, $\w \phi \in H^{-s}(\R^d)$, $n\ge s$ and
    the quasi-projection operator $Q_j(f)$ is well-defined by~(\ref{QjDef}) for functions $f \in H^s(\R)$. We say that $Q_j$ provides 
    simultaneous density order $(s,n)$ in $H^s(\R^d)$, if
    for all functions $f \in H^n(\R^d)$
    \begin{equation}
        \lim_{j\to +\infty}\norm{M^{*-j}}^{-n} \sum_{r=0}^{s} \norm{M^{*j}}^{-r}
        \left| f - \sum\limits_{k\in\Z^d}\langle f, \w\phi_{jk}\rangle\phi_{jk} \right|_{H^{r}}  = 0.
        \label{fSimult}  
    \end{equation}
    \end{defi}

Note that simultaneous density order $(s,n)$ for $Q_f$ implies 
simultaneous approximation order $(s,n)$. In the statements below we establish sufficient and
necessary (for isotropic dilation matrices) conditions to the fact that $Q_j$ provides 
simultaneous density order $(s,n)$.

\begin{thm}
     Let $s \in \N_0$, $\phi \in H^s(\R^d)$ and 
     $\w \phi \in H^{-s}(\R^d)$
     have compact supports, $n\ge s.$ 

    1. Suppose that conditions 
        \begin{equation}
        \lim\limits_{t\rightarrow \nul} \frac{1}{|t|^{n}}\left|1-
		\h \phi(t) \ol{\h {\w \phi}(t)} \right|
         = 0,
         \quad
         \lim\limits_{t \rightarrow \nul} \frac{1}{|t|^{2n}} [\h\phi, \h\phi]_r^{\ne 0}(t) = 0
         \label{fSDO}
        \end{equation}
        hold for all $0 \leq r \leq s$.
    Then the quasi-projection operator $Q_j$ provides simultaneous density order $(s,n)$ in $\mathcal S$.

    2. Suppose that $M$ is an isotropic dilation matrix. The quasi-projection operator 
    $Q_j$ provides  simultaneous density order $(s,n)$ in $\mathcal S$ if and only if conditions~(\ref{fSDO}) hold.
    \label{density}
\end{thm}
\noindent\textbf{Proof.}
Let us fix $j\in \Z$ and denote  
    \[
    G_j(\xi):=\sum\limits_{l \in \Z^d}\h f(M^{*j}(\xi + l))\ol {\h{\w \phi}(\xi + l)} = [\h f(M^{*j} \cdot), \h {\w \phi}](\xi),\quad \xi\in \R^d. 
    \]
	By Lemma~\ref{lemGjMain} function $G_j$ is bounded on ${\mathbb T}^d$
    and using Remark~\ref{rem1} it can be checked that
    $\langle f, \w{\phi}_{jk}\rangle=m^{\frac{j}{2}}\h G_j(k)$, 
	$k \in\Z^d$. Hence,
    $$
    \left(\smlm_{k \in \Z^d} \langle f, \w\phi_{jk}\rangle{\phi_{jk}} \right)^{\wedge} = 
    \sum\limits_{k \in \Z^d}\h G_j(k)e^{2\pi i(k,M^{*-j}\xi)}\h \phi(M^{-*j}\xi) = G_j(M^{*-j}\xi)\h \phi(M^{*-j}\xi).
    $$
    Also, it is convenient to use notation
    $G_j^{\neq\nul}(\xi) := G_j(\xi) - \h f(M^{*j}\xi) \ol{\h{\w\phi}(\xi)}.$
    
	Item 1. Consider one term in the left-hand side of~(\ref{fSimult})
	\begin{align*}
		\frac{\norm{M^{*j}}^{-2r}}{\norm{M^{*-j}}^{2n}} 
        &\left| f-\smlm_{k \in \Z^d} \langle f, \w\phi_{jk}\rangle{\phi_{jk}} \right|_{H_r}^2 
        \\
        & =\frac{\norm{M^{*j}}^{-2r}}{\norm{M^{*-j}}^{2n}} \int\limits_{\R^d}|\xi|^{2r}\left|\h f(\xi)-
        G_j(M^{*-j}\xi)\h \phi(M^{*-j}\xi)\right|^2\dd \xi :=I_1 + I_2.
	\end{align*}  
	Here, the last integral is represented as a sum of integral $I_1$ over $M^{*j}\T^d$ and integral $I_2$ over $\R^d \setminus M^{*j}\T^d$. 
    
    1. Consider integral $I_1$, taking into account that
    $G_j = \h f(M^{*j}\cdot) \ol{\h{\w\phi}} + G_j^{\neq\nul}$.
    Then
    \begin{align*}
    I_1 & = \int\limits_{M^{*j}\T^d}
    \norm{M^{*j}}^{-2r}
    |\xi|^{2r}
    \left|
    \frac{\h f(\xi) \left(1-\h\phi \ol{\h{\w\phi}} \right)(M^{*-j}\xi)}{\norm{M^{*-j}}^{n}}
    - \frac{G^{\neq \nul}_j(M^{*-j}\xi)  \h \phi(M^{*-j}\xi)}{\norm{M^{*-j}}^{n}} \right|^2 \dd \xi
    \\
    & = \int\limits_{\R^d}
    \norm{M^{*j}}^{-2r}
    |\xi|^{2r}
    \left|\frac{\h f(\xi) \left(1-\h\phi \ol{\h{\w\phi}} \right)(M^{*-j}\xi)}{\norm{M^{*-j}}^{n}}\right|^2 \chi_{M^{*j}\T^d}(\xi)  \dd \xi + o(1)
    \quad \text{as} \ j \to + \infty.
    \end{align*}  
    This can be shown by expanding the brackets of the squared module 
    according to the template 
    $|a-b|^2 = |a|^2 - 2 {\mathcal Re} (a b) + |b|^2$.
    First, integral over the 3rd term in this template tends to zero as $j \to + \infty$.
    Indeed, note that
    $\h \phi(M^{*-j}\cdot)$ is bounded on $M^{*j}\T^d$,
    $\norm{M^{*j}}^{-2r} |\xi|^{2r} \le |M^{*-j}\xi|^{2r} \le 1$ for
    $\xi \in M^{*j}\T^d$ and by Lemma~\ref{lemGjMain} we can choose $N > 0$ such that relations
    $\frac{G^{\neq \nul}_j(M^{*-j}\xi) }
    {\norm{M^{*-j}}^{n} } \le C  \|M^{*-j}\|^N$ and~(\ref{fMjEstimation2}) are valid.
    Hence, as $j \to + \infty$
    \begin{equation}
    \int\limits_{\R^d}
    \norm{M^{*j}}^{-2r}
    |\xi|^{2r}
    \left|
    \frac{G^{\neq \nul}_j(M^{*-j}\xi)  \h \phi(M^{*-j}\xi)}{\norm{M^{*-j}}^{n}} 
    \right|^2\chi_{M^{*j}\T^d}     \dd \xi
    \le C \|M^{*-j}\|^{2N} m^j \to 0.
    \label{fIneq0}
    \end{equation}
    Second, integral over the 2nd term $2 {\mathcal Re} (a b)$ in the above template also tends to zero as $j \to + \infty$ by the Cauchy inequality, taking into account relation~(\ref{fIneq0}) and estimation~(\ref{fIneq1}) from below.
    Applying
    the Lebesgue dominated convergence theorem we, furthermore, can show that $\lim\limits_{j\to+\infty} I_1=0$.
    Indeed, by~(\ref{fSDO}) there exists a constant $C>0$ such that $\left|1-\h\phi \ol{\h{\w\phi}}\right|(\xi) \le C |\xi|^n$ for any $\xi \in \T^d$.  
    Therefore, for any $\xi \in \R^d$
    \begin{equation}
    \norm{M^{*j}}^{-2r}
    |\xi|^{2r}
    \left|\frac{\h f(\xi) \left(1-\h\phi \ol{\h{\w\phi}} \right)(M^{*-j}\xi)}{\norm{M^{*-j}}^{n}}\right|^2 \chi_{M^{*j}\T^d}(\xi)
    \le C^2 \left|\h f(\xi)\right|^2 |\xi|^{2n},
    \label{fIneq1}
    \end{equation}
    where the right-hand side is summable.
    Hence,
    \begin{align*}
    \lim_{j\to +\infty} I_1 = &
\int\limits_{\R^d}
		\left|\h f(\xi)\right|^2 |\xi|^{2r} \lim_{j\to +\infty} 
        \frac{\left|1-\h\phi \ol{\h{\w\phi}} \right|^2(M^{*-j}\xi)}
        {\norm{M^{*-j}}^{2n} \norm{M^{*j}}^{2r}} \chi_{M^{*j}\T^d}(\xi)
        \dd\xi 
    \\
    = & \int\limits_{\R^d}
		\left|\h f(\xi)\right|^2 |\xi|^{2n}  \lim\limits_{j\to +\infty} 
        \frac{\left|1-\h\phi \ol{\h{\w\phi}} \right|^2(M^{*-j}\xi)}
        {\|M^{*-j}\|^{2(n-r)}|\xi|^{2(n-r)}}
        \frac{\norm{M^{*-j}}^{-2r}}{\norm{M^{*j}}^{2r} }
        \dd\xi.
    \end{align*} 
    Since $\norm{M^{*-j}}\norm{M^{*j}} \ge 1$, then 
    $\frac{\norm{M^{*-j}}^{-2r}}{\norm{M^{*j}}^{2r} } \le 1$  and  for any $\xi\in\R^d$
    $$
        \frac{\left|1-\h\phi \ol{\h{\w\phi}} \right|(M^{*-j}\xi)}
        {\|M^{*-j}\|^{n-r}|\xi|^{n-r}} \frac{\norm{M^{*-j}}^{-r}}{\norm{M^{*j}}^{r} } \le
        \frac{\left|1-\h\phi \ol{\h{\w\phi}} \right|(M^{*-j}\xi)}
        {|M^{*-j}\xi|^{n-r}} \to 0 \quad \text{as} \ j \to +\infty
    $$
    by conditions~(\ref{fSDO}), which yields that
    $\lim\limits_{j\to+\infty} I_1=0$.

    2. Consider integral $I_2$. After splitting it into integrals over $M^{*j}\T^d+k$, $k\in\Z^d$, $k\neq \nul$ and changing the variables,
    we get that
    \begin{align*}
    I_2 &  = \frac{\norm{M^{*j}}^{-2r}}{\norm{M^{*-j}}^{2n}} 
    \int\limits_{\R^d \setminus M^{*j}\T^d}|\xi|^{2r}
    \left|\h f(\xi)-
        G_j(M^{*-j}\xi)\h \phi(M^{*-j}\xi) \right|^2\dd \xi
    \\
    & = \sum_{k\neq\nul}\frac{\norm{M^{*j}}^{-2r}}{\norm{M^{*-j}}^{2n}}
\int\limits_{M^{*j}\T^d}|\xi+M^{*j}k|^{2r} 
		\left|\h f(\xi+M^{*j}k) -
        G_j(M^{*-j}\xi)\h{\phi}(M^{*-j}\xi+k) \right|^2\dd\xi.
    \end{align*} 
    Again, expanding the brackets of the squared module according to the template 
    $|a-b|^2 = |a|^2 - 2 {\mathcal Re} (a b) + |b|^2$,
    the above integral $I_2$ can be represented as a sum of three terms $I_{21}+I_{22}+I_{23}$. 
    Consider each term separately.
    
    First, since $f\in {\cal S}$ and for any $N>0$ there exists a constant $C$ such that $|\h f(\xi)| \le \frac{C}{|\xi|^{N+r+n}}$ for any $\xi \in \R^d \setminus \T^d$, then 
    \begin{align*}
    I_{21} & =    
\int\limits_{M^{*j}\T^d}
\sum_{k\neq\nul}\frac{\norm{M^{*j}}^{-2r}}{\norm{M^{*-j}}^{2n}}|\xi+M^{*j}k|^{2r} 
		\left|\h f(\xi+M^{*j}k)\right|^2 \dd\xi = 
    \\
    & = \int\limits_{\R^d \setminus M^{*j}\T^d}
    \frac{\norm{M^{*j}}^{-2r}}{\norm{M^{*-j}}^{2n}}|\xi|^{2r} 
		\left|\h f(\xi)\right|^2 \dd\xi = 
       m^j \int\limits_{\R^d \setminus \T^d}
    \frac{\norm{M^{*j}}^{-2r}}{\norm{M^{*-j}}^{2n}}|M^{*j}\xi|^{2r} 
		\left|\h f(M^{*j}\xi)\right|^2 \dd\xi 
        \\
    & \le m^j \int\limits_{\R^d \setminus \T^d}
    \frac{C^2 |\xi|^{2r} \dd\xi }{\norm{M^{*-j}}^{2n} |M^{*j}\xi|^{2(N+n+r)}}
    \le m^j \int\limits_{\R^d \setminus \T^d}
    \frac{C^2 \norm{M^{*-j}}^{2(N+r)}  \dd\xi }{|\xi|^{2(N+n)}},
    \end{align*} 
   where $N>0$ is taken such that relation~(\ref{fMjEstimation2}) is valid and integral $\int\limits_{\R^d \setminus \T^d} \frac{\dd\xi}{|\xi|^{2(N+n)}}$ is convergent. Hence, 
    $\lim\limits_{j\to+\infty} I_{21}=0$.
   
   Second,
    \begin{align*}
    I_{22} & = 
    \sum_{k\neq\nul}\frac{\norm{M^{*j}}^{-2r}}{\norm{M^{*-j}}^{2n}}
\int\limits_{M^{*j}\T^d}|\xi+M^{*j}k|^{2r} 
		2{\mathcal Re}\left(\h f(\xi+M^{*j}k) \ol{G_j(M^{*-j}\xi)\h{\phi}(M^{*-j}\xi+k)}\right) \dd\xi 
    \\ & =
\int\limits_{M^{*j}\T^d}
		2{\mathcal Re}\left( \ol{G_j(M^{*-j}\xi)} 
        \sum_{k\neq\nul}  |\xi+M^{*j}k|^{2r} \h f(\xi+M^{*j}k)
        \h{\phi}(M^{*-j}\xi+k) \frac{\norm{M^{*j}}^{-2r}}{\norm{M^{*-j}}^{2n}}\right) \dd\xi.
    \end{align*}
    Lemma~\ref{lemGjMain} can be applied to state that
    $$
    |[\h g(M^{*j}\cdot),\h\phi]^{\neq\nul}(M^{*-j}\xi)| \le C \|M^{*-j}\|^{N+2n}
    $$
    for $\xi \in M^{*j}\T^d$ where $\h g = |\cdot|^{2r}\h f$  and 
    $N$ is such that  relation~(\ref{fMjEstimation2}) is valid. Also, note that $G_j$ is bounded.  
    Hence, $|I_{22}| \le C \|M^{*-j}\|^{N} m^j$ and 
    $\lim\limits_{j\to+\infty} I_{22}=0$.

    Third, denote
    $$
    A (M^{*-j}\xi)= \sum_{k\neq\nul}\norm{M^{*j}}^{-2r}|\xi+M^{*j}k|^{2r} 
		\left|\h{\phi}(M^{*-j}\xi+k) \right|^2,
    $$
    then
    \begin{align*}
    I_{23} & =    
\int\limits_{M^{*j}\T^d}
\frac{\left| G_j(M^{*-j}\xi) \right|^2}{\norm{M^{*-j}}^{2n}}
		 A (M^{*-j}\xi) \dd\xi = 
    \\
    & = \int\limits_{M^{*j}\T^d}
\frac{\left| G^{\neq\nul}_j(M^{*-j}\xi) \right|^2}{\norm{M^{*-j}}^{2n}}
		 A (M^{*-j}\xi)  \dd\xi+ \int\limits_{M^{*j}\T^d}
\frac{\left| \h f(\xi) \h{\w\phi}(M^{*-j}\xi) \right|^2}{\norm{M^{*-j}}^{2n}}
		 A (M^{*-j}\xi)  \dd\xi 
         \\
         & \hspace{1cm} +
\int\limits_{M^{*j}\T^d}
2{\mathcal Re}\left(\frac
{G^{\neq\nul}_j(M^{*-j}\xi) }{\norm{M^{*-j}}^{2n}}  
\ol{\h f(\xi)} \h{\w\phi}(M^{*-j}\xi) \right)
		 A (M^{*-j}\xi)  \dd\xi = J_1+J_2+J_3.
     \end{align*}
    Note that $A( M^{*-j}\xi) \le [\h\phi, \h\phi]_r^{\neq\nul}(M^{*-j}\xi)$ and, hence, it is bounded on $M^{*j}\T^d$. By Lemma~\ref{lemGjMain} 
    $|G^{\neq\nul}_j(M^{*-j}\xi)| \le C \|M^{*-j}\|^{N+n}$, where 
    $N$ is taken such that relation~(\ref{fMjEstimation2}) is valid.
    Thus, $\lim\limits_{j \to + \infty} J_1 = 0$.
    By similar arguments $\lim\limits_{j \to + \infty} J_3 = 0$ also.
    It remains to consider
    $$
    J_2 = \int\limits_{M^{*j}\T^d}
\left| \h f(\xi) \right|^2 |\xi|^{2n} 
\left|\h{\w\phi}(M^{*-j}\xi) \right|^2
		 \sum_{k\neq\nul}\frac{\norm{M^{*j}}^{-2r}|\xi+M^{*j}k|^{2r} 
		\left|\h{\phi}(M^{*-j}\xi+k) \right|^2 }{\norm{M^{*-j}}^{2n}|\xi|^{2n}}
        \dd\xi.
    $$
    By the Lebesgue dominated convergence theorem
    $$
    \lim\limits_{j \to + \infty} J_2 = 
    \int\limits_{\R^d}
\left| \h f(\xi) \right|^2 |\xi|^{2n} 
\left|\h{\w\phi}(0) \right|^2
		 \lim\limits_{j \to + \infty}\sum_{k\neq\nul}\frac{\norm{M^{*j}}^{-2r}|\xi+M^{*j}k|^{2r} 
		\left|\h{\phi}(M^{*-j}\xi+k) \right|^2 }{\norm{M^{*-j}}^{2n}|\xi|^{2n}}
        \dd\xi,
    $$
    since the right-hand side of the following inequality
    $$
    \left|\h{\w\phi}(M^{*-j}\xi) \right|^2\sum_{k\neq\nul}\frac{\norm{M^{*j}}^{-2r}|\xi+M^{*j}k|^{2r} 
		\left|\h{\phi}(M^{*-j}\xi+k) \right|^2 }{\norm{M^{*-j}}^{2n}|\xi|^{2n}} \le 
       \left|\h{\w\phi}(M^{*-j}\xi) \right|^2\frac{ [ \h\phi, \h\phi]^{\neq\nul}_r(M^{*-j}\xi)}{|M^{*-j}\xi|^{2n}}
    $$
    is bounded on $M^{*j} \T^d$, and, hence, we cen get a summable majorant.
    Also, by conditions~(\ref{fSDO}) this right-hand side tends to 0 as
    $j\to + \infty$ and $\lim\limits_{j \to + \infty} J_2 =0$.

    Overall,  $\lim\limits_{j \to + \infty} I_{23} =0$ and, therefore, 
    $\lim\limits_{j \to + \infty} I_{2} =0$, which yields that the quasi-projection operator $Q_j$ provides simultaneous density order $(s,n)$.

    Item 2. It remains to prove the necessity. Again, consider one term in the left-hand side of~(\ref{fSimult}) and by the above analysis we get
    \begin{align*}
		0 &= \lim\limits_{j \to + \infty}  \frac{\norm{M^{*j}}^{-2r}}{\norm{M^{*-j}}^{2n}} 
        \left| f-\smlm_{k \in \Z^d} \langle f, \w\phi_{jk}\rangle{\phi_{jk}} \right|_{H_r}^2 
        \\
        & = \lim\limits_{j\to +\infty}  \int\limits_{M^{*j}\T^d}
		\left|\h f(\xi)\right|^2 |\xi|^{2n}  
        \left(
        \frac{\left|1-\h\phi \ol{\h{\w\phi}} \right|^2(M^{*-j}\xi)}
        {\|M^{*-j}\|^{2(n-r)}|\xi|^{2(n-r)}}
        \frac{\norm{M^{*-j}}^{-2r}}{\norm{M^{*j}}^{2r} }
    \right.+
        \\
         & \hspace{2cm} \left. +
\left|\h{\w\phi}(M^{*-j}\xi) \right|^2
		 \sum_{k\neq\nul}\frac{\norm{M^{*j}}^{-2r}|\xi+M^{*j}k|^{2r} 
		\left|\h{\phi}(M^{*-j}\xi+k) \right|^2 }{\norm{M^{*-j}}^{2n}|\xi|^{2n}}
        \right)
        \dd\xi.
	\end{align*}  
    Denote the sum of two terms inside the parentheses under the integral by $T_r$. Consider $r=0$. Taking into account
    inequality~(\ref{fIsotropicM}) we have the following estimate from below
    $$
    T_0 \ge \frac{\left|1-\h\phi \ol{\h{\w\phi}} \right|^2(M^{*-j}\xi)}
        {\|M^{*-j}\|^{2n}|\xi|^{2n}}
         \ge 
         \frac{\left|1-\h\phi \ol{\h{\w\phi}} \right|^2(M^{*-j}\xi)}
        {\|M^{*-j}\|^{2n} \|M^{*j}\|^{2n} |M^{*-j}\xi|^{2n}} \ge 
        \frac{\left|1-\h\phi \ol{\h{\w\phi}} \right|^2(M^{*-j}\xi)}
        {C_2^{4n} |M^{*-j}\xi|^{2n}}.
    $$
    Suppose that the first condition in~(\ref{fSDO}) does not hold, i.e. there exists $\varepsilon_0>0$ such that for any $\delta < \frac 12$ there exists 
    $t\in B_{\delta}(\nul)$, $t\neq\nul$, such that $\frac{1}{|t|^{n}}\left|1-
		\h \phi(t) \ol{\h {\w \phi}(t)} \right| \ge \varepsilon_0$. 
        By continuity in some small neighborhood $U\subset \T^d$ of $t$ inequality $|\cdot|^{-n}\left|1-
		\h \phi \ol{\h {\w \phi}} \right| \ge \frac{\varepsilon_0}{2}$ holds. 
        Hence,
        $$
         0 \ge \lim\limits_{j\to +\infty} \int\limits_{M^{*j}\T^d}
        \left| \h f(\xi) \right|^2 |\xi|^{2n} 
        \frac{\left|1-\h\phi \ol{\h{\w\phi}} \right|^2(M^{*-j}\xi)}
        {| M^{*-j}\xi|^{2n}}\dd\xi \ge 
        \lim\limits_{j\to +\infty} \int\limits_{M^{*j}U}
        \left| \h f(\xi) \right|^2 |\xi|^{2n} 
        \frac{\varepsilon^2_0}{4}\dd\xi > 0,
        $$
        which gives a contradiction. Thus, $\lim\limits_{t\rightarrow \nul} \frac{1}{|t|^{n}}\left|1-
		\h \phi(t) \ol{\h {\w \phi}(t)} \right|
         = 0$ and, 
        in particular, 
         $1-
		\h \phi(\nul) \ol{\h {\w \phi}(\nul)} = 0$. Thus, 
        $\h {\w \phi}(\nul) \neq 0$. Together with this fact by similar arguments we can show that 
        $\lim\limits_{t \rightarrow \nul} \frac{1}{|t|^{2n}} [\h\phi, \h\phi]_r^{\ne \nul}(t) = 0$ for $r=0,...,s.$ $\qed$  

        Theorem~\ref{density} is stated for functions $f$ from the Schwartz class ${\mathcal S}$. However, by density arguments this theorem can be extended to functions from the Sobolev space $H^n(\R^d)$. The proof is based on the Moore-Osgood theorem on exchanging limits.
        
    \begin{lem}~\emph{\cite[Theorem 7.11]{Rudin}}
    \label{Moora}
		Let $\{a_{i,j}\}_{i,j\in\N}$ be a double numerical  sequence. If
		\begin{enumerate}
			\item $a_{i,j}$ converges uniformly to $a_j\in \R$ as $i \to +\infty$, i.e.
			$\lim\limits_{i \rightarrow +\infty}\sup\limits_{j} |a_{i,j} - a_j| = 0$;
			\item for any fixed $i$ sequence $a_{i,j}$ converges to $a_i\in \R$ as $i\to +\infty$,
		\end{enumerate}
		then there exists a limit $\lim\limits_{j\rightarrow +\infty} a_j$ and
		$
			\lim\limits_{i \rightarrow +\infty}\lim\limits_{j \rightarrow +\infty} a_{i,j} =
			\lim\limits_{j \rightarrow +\infty}\lim\limits_{i \rightarrow +\infty} a_{i,j}.
		$
	\end{lem}

	\begin{col}
        Item 1 and Item 2 of Theorem~\ref{density}
        can be stated for functions $f$ from $H^s(\R^d)$. 
	\end{col}
	\textbf{Proof.} 
    Consider Item 1. Suppose $f\in H^n(\R^d)$. 
	Since $\mathcal S$ is dense in $H^n(\R^d)$, there exists a sequence $\{f_i\}_{i\in\N} \subset {\mathcal S}$ 
    such that $\| f - f_i \|_{H^n} < \frac{1}{i}$ for $i\in \N$. Fix $0 \le r \le s$ and denote
	\[
		a_{i,j} = \| M^{*-j} \|^{-n}  \| M^{*j} \|^{-r} \left| f_i - Q_j(f_i) \right|_{H^r}.
	\]
	By Theorem~\ref{density} $\lim\limits_{j\rightarrow +\infty} a_{i,j} = 0$ for any $i\in\N$. 
    Conditions~(\ref{fSDO}) imply
    conditions~(\ref{cond6}) and, therefore, by Theorem \ref{theoMainAppr}  we get
	\begin{align*}
		& \lim\limits_{i \rightarrow +\infty} \sup\limits_{j} \| M^{*-j} \|^{-n}  \| M^{*j} \|^{-r} \Big| |f - Q_j(f)|_{H^r} - |f_i - Q_j(f_i)|_{H^r} \Big| 
        \\
		 &\hspace{2cm}\leq \lim\limits_{i \rightarrow +\infty} \sup\limits_{j} \| M^{*-j} \|^{-n}  \| M^{*j} \|^{-r} \Big| (f - f_i) - Q_j(f-f_i) \Big|_{H^r} 
         \\
		 &\hspace{2cm}\leq \lim\limits_{i \rightarrow +\infty} \sup\limits_{j} C |f - f_i|_{H^n} = 0.
	\end{align*} 
    Therefore, $a_{i,j} \rightarrow a_j = \| M^{*-j} \|^{-n}  \| M^{*j} \|^{-r} \left| f - Q_j(f) \right|_{H^r}$ as $i\to+\infty$ uniformly by $j\in \N$.
	By Lemma~\ref{Moora} we get
	$
			0 = \lim\limits_{i \rightarrow +\infty}\lim\limits_{j \rightarrow +\infty} a_{i,j} =
		\lim\limits_{j \rightarrow +\infty}\lim\limits_{i \rightarrow +\infty} a_{i,j}	
	$, i.e.

	\[
		\lim\limits_{j\rightarrow +\infty}  \| M^{*-j} \|^{-n}  \| M^{*j} \|^{-r} \left| f - Q_j(f) \right|_{H^r} = 0. 
	\]
    Item 2 is clear by Item 1 and since ${\mathcal S} \subset H^n(\R^d)$. $\qed$


\section{Approximation by wavelet systems}
\label{sec4}

A direct application of the established results is related to studying the approximation properties of wavelets.
    In what follows, we assume that $M$ is an isotropic dilation matrix.
    The construction of wavelet systems starts from a suitable pair of scaling functions $\phi$ and $\widetilde\phi$. 
    A \textit{scaling function} $\phi$ is a solution of the scaling equation
    \begin{equation}
    \phi(x) = m \sum_{k\in\mathbb{Z}^d} a(k) \phi(M x + k),
    \label{fRE}
    \end{equation}
    where $x\in\mathbb{R}^d$, $a=\{a(k)\}_{k\in\mathbb{Z}^d}\in \ell_0(\mathbb{Z}^d)$.
Applying the Fourier transform to the scaling equation~\eqref{fRE}, we obtain
    \begin{equation}
    \widehat\phi(\xi) = \widehat a
    \left(M^{*-j}\xi\right)
    \widehat\phi\left(M^{*-j}\xi\right),
    \qquad \xi \in \mathbb{R}^d.
    \label{fRE_Fourier}
    \end{equation}
A sequence $a\in \ell_0(\mathbb{Z}^d)$ with $\widehat a(\nul)=1$ is called a \textit{mask}.
It is known that for a mask $a$ the solution of the scaling equation~\eqref{fRE} is unique up to a constant factor; moreover, $\phi$ is a tempered distribution with compact support (see, e.g.,~\cite[Theorem 2.6.4]{KPS}). 
In order to obtain $s\in \R$ such that $\phi$ belongs to $H^s(\R^d)$ we can compute 
the Sobolev smoothness exponent for $\phi$, which is defined by
$$
\nu_2(\phi)=
        \sup \left\{ \nu\in\R:
         \int\limits_{\R^d} |\h \phi (\xi)|^2 (1+ |\xi|^2)^{\nu} 
         d \xi < +\infty\right\}.
$$
Thus, $\phi \in H^s(\R^d)$ for $s< \nu_2(\phi).$
This quantity can be estimated by effective algorithm developed in~\cite{HanSmooth}.
Below, we assume that the scaling functions are normalized, i.e. $\widehat\phi(\nul)=1$.

A standard scheme for constructing so called MRA-based wavelet systems relies on the unitary extension principle and its modifications (see, e.g.,~\cite{RS2},~\cite{KPS}). Consider the following setup.
Let $a,\widetilde a\in\ell_0(\mathbb{Z}^d)$ be two masks, and 
let $\phi,\widetilde\phi$ be the corresponding scaling functions 
such that $\nu_2(\phi) > s$ and $\nu_2(\w\phi) > - s$ for some $s \ge 0$.
Hence, $\phi \in H^s(\R^d)$ and $\w\phi \in H^{-s}(\R^d)$.
For convenience, we set
\[
b_{1} := a, \qquad \widetilde b_{1} := \widetilde a.
\]
Let $D(M)$ be the set of digits of dilation matrix $M$, for instance, we can set 
$D(M) = M [0,1)^d \cap \Z^d$, the cardinality of $D(M)$ is equal to $m$ (for details, see, e.g.,~\cite{KPS}).
Suppose that we can find trigonometric polynomials $\widehat b_{\mu}(\xi)$, $\widehat{\widetilde b}_{\mu}(\xi)$, $\mu=2,\dots,u$, with $u\ge m$, called \emph{wavelet masks}, such that the following $u\times m$ matrices
\[
{\mathcal L}(\xi):=
\Bigl\{\widehat b_{\mu}\bigl(\xi+M^{*-j}s\bigr)\Bigr\}_{\mu=1,\dots,u}^{s\in D(M^*)},
\qquad
\widetilde{\mathcal L}(\xi):=
\Bigl\{\widehat{\widetilde b}_{\mu}\bigl(\xi+M^{*-j}s\bigr)\Bigr\}_{\mu=1,\dots,u}^{s\in D(M^*)},
\]
satisfy
\begin{equation}
    {\mathcal L}(\xi)^* \widetilde{\mathcal L}(\xi) \equiv I_m
    \label{fcalL2},
\end{equation}
i.e., the columns of these matrices are pairwise biorthogonal.
The \textit{wavelet functions} $\psi_{\mu}$, $\widetilde\psi_{\mu}$, $\mu = 2,\dots, u$, can be defined through their Fourier transforms as
\begin{equation}
    \widehat{\psi_{\mu}}(\xi)=
    \widehat b_{\mu}\left(M^{*-j}\xi\right) \widehat\phi\left(M^{*-j}\xi\right), \quad
    \widehat{\widetilde\psi_{\mu}}(\xi)=
    \widehat{\widetilde b}_{\mu}\left(M^{*-j}\xi\right) \widehat{\widetilde\phi}\left(M^{*-j}\xi\right).
    \label{PsiDef}
\end{equation}
Note that $\psi_{\mu} \in H^s(\R^d)$, $\widetilde\psi_{\mu} \in H^{-s}(\R^d)$ and have compact supports. 
The obtained pair of systems $\{\psi_{\mu,j,k}\}$, $\{\widetilde\psi_{\mu,j,k}\}$
is called a \textit{MRA-based wavelet system} in $(H^s(\R^d), H^{-s}(\R^d))$  generated by the pair of scaling functions
$\phi, \widetilde\phi$.

Quasi-projection operators are related to wavelet systems via a special relation known as the perfect reconstruction property 
(see, e.g.~\cite[Lemma 4.3.1]{KPS}).

\begin{lem}
    \label{lemPRP}
    Let $s \ge 0$, functions
		$\phi \in H^s(\R^d)$ and $\w \phi \in H^{-s}(\R^d)$ 
        have compact supports, $f \in H^s(\R^d)$. Suppose $\{\psi_{\mu,j,k}\}$, $\{\widetilde\psi_{\mu,j,k}\}$ is a pair of MRA-based wavelet system generated by $\phi, \widetilde\phi$, $\mu=2,\dots,u$, $u\ge m$, $j,j'\in\mathbb{Z}$, $j'>j$.
    Then 
    \begin{equation}
        \sum_{k\in\mathbb{Z}^d} 
        \langle f, \widetilde\phi_{j'k}\rangle
        \phi_{j'k} -
        \sum_{k\in\mathbb{Z}^d} 
        \langle f, \widetilde\phi_{jk}\rangle
        \phi_{jk}
        =
        \sum_{i=j}^{j'-1}
        \sum_{\mu=2}^{u}
        \sum_{k\in\mathbb{Z}^d}
        \langle f, \widetilde\psi_{\mu,i,k}\rangle
        \psi_{\mu,i,k}.
        \label{fLemPRP}
    \end{equation}
\end{lem}

In~\cite{Han} the notion of dual wavelet frames in a pair of dual Sovolev spaces 
    was introduced and studied for the case $M=2 I_d$ and later these results were extended
    in~\cite{Zhang18} for the case of isotropic dilation matrices.
    In particular, if a pair of wavelet systems $(\{\phi_{0k}\}, \{\psi^s_{\mu,j,k}\})$, $(\{\w\phi_{0k}\}, \{\widetilde\psi^{-s}_{\mu,j,k}\})$ (here $j \ge 0$ and  
    $\psi^s_{\mu,j,k} = m^{js}\psi_{\mu,j,k}$, $\widetilde\psi^{-s}_{\mu,j,k} = 
    m^{-js}\widetilde\psi_{\mu,j,k}$) forms a dual wavelet frame in 
    $(H^s(\R^d), H^{-s}(\R^d))$, then each system is a frame $H^s(\R^d)$ and 
    $H^{-s}(\R^d)$, respectively, and 
    $$
    f = \sum\limits_{k \in \Z^d} \langle f, \w{\phi}_{0k} \rangle \phi_{0k} + 
    \sum_{j=0}^{+\infty}
        \sum_{\mu = 2}^{u}
        \sum_{k\in\mathbb{Z}^d}
        \langle f, \widetilde\psi_{\mu,j,k}\rangle
        \psi_{\mu,j,k}, 
        \quad \text{for } f\in H^s(\R^d),
    $$
    where the series converges unconditionally in $H^s(\R^d).$
    The following statement was established in~\cite[Theorem 4.1]{Zhang18}.
\begin{thm}
    Let $s\ge 0$. Assume that $\{\psi_{\mu,j,k}\}$, $\{\widetilde\psi_{\mu,j,k}\}$ is a MRA-based wavelet system  
in $(H^s(\R^d), H^{-s}(\R^d))$ generated by $\phi, \widetilde\phi$. If there exist numbers 
$\alpha \ge 0$ and $\w\alpha \ge 0$ with $\alpha > -s$ and $\w\alpha > s$, such that the following
conditions on vanishing moments hold
$$
\h {b_{\mu}}(\xi) = O(|\xi|^{\alpha}), \quad 
\h {{\w b}_{\mu}}(\xi) = O(|\xi|^{\w\alpha}) \quad
\text{as } \xi \to 0, \quad \mu=2,\dots,u,
$$
then the MRA-based wavelet system produces a pair of wavelet dual frame in 
a pair of dual Sobolev spaces
    $(H^s(\R^d), H^{-s}(\R^d))$.
    \label{fDualWF}
\end{thm}    

Let us discuss how to provide high order of simultaneous approximation for dual wavelet frames from Theorem~\ref{fDualWF}.
Suppose that $s \in \N$, so we can assume $\alpha = 0$ and, hence, conditions $\h {b_{\mu}}(\xi) = O(1)$
as $\xi \to 0$ for $\mu=2,\dots,u$ always hold. Suppose also that we provide somehow 
\textit{vanishing moments} conditions of order $n$ for $\h{\w b_{\mu}}$, i.e.
$\h {{\w b}_{\mu}}(\xi) = O(|\xi|^{n})$ as $\xi \to 0$ for $\mu=2,\dots,u$.
Let us show that then conditions~(\ref{cond6}) for scaling functions $\phi$ and $\w\phi$
are valid.
For this purpose we need to recall the sum rule condition for mask $a$ and 
the compatibility condition for a pair of masks $a$ and $\w a$.
We say that mask $a$ satisfies the \textit{sum rule} of order $n$, if
$\h a (M^{*-j}s + \xi) = O(|\xi|^n)$ as $\xi \to 0$ for all 
$s \in D(M^*)\setminus\{\nul\}.$ \textit{Compatibility condition} of order $n$ for a pair of masks $a$ and $\w a$ means that $1 - \h a(\xi) \overline{\h {\w a}(\xi)} = O(|\xi|^n)$ as $\xi \to 0$.
From~(\ref{fcalL2}) we have
$$
\h a(\xi + M^{*-1}s) \h {\w a}(\xi) + \sum_{\mu=2}^u \h{b_{\mu}}(\xi + M^{*-1}s) 
\h{\w b_{\mu}}(\xi) = \delta_{\nul,s}.
$$
Since $\h b_{\mu}$, $\mu=2,\dots,u$ are bounded and $\h{\w a}(\nul)=1$, 
then vanishing moments conditions of order $n$ for $\h{\w b_{\mu}}$ implies
the sum rule of order $n$ for $a$ and the compatibility condition of order $n$
for the pair of masks $a$ and $\w a$.

Next, let us show that the compatibility of masks implies the compatibility of scaling functions, i.e. the first inequality in~(\ref{cond6}) is justified. Indeed,
$1 - \h a(\xi) \overline{\h {\w a}(\xi)} = O(|\xi|^n)$ implies that 
$\h\phi(\xi) \overline{\h {\w\phi}(\xi)}\left(1 - \h a(\xi) \overline{\h {\w a}(\xi)}\right) = O(|\xi|^n)$ as $\xi \to 0$ and, hence,
$$
\left(1-\h\phi(M^*\xi) \overline{\h {\w\phi}(M^*\xi)}\right)-
\left(1 - \h\phi(\xi) \overline{\h {\w\phi}(\xi)}\right) = O(|\xi|^n)
$$
as $\xi \to 0$.
This is only possible if $1 - \h\phi(\xi) \overline{\h {\w\phi}(\xi)} = O(|\xi|^n)$ as $\xi \to 0$ (see, e.g., \cite[Section 2]{ChuiJiang} and~\cite[Section 5]{KMulti} for relevant details).

Finally, let us show that the sum rule of order $n$ for $a$ implies the second inequality in~(\ref{cond6}). 
It is known (see, e.g.,~\cite[Theorem 3.2.8]{Jetter2001}) that, if mask $a$ satisfies the sum rule of order $n$, then the corresponding
scaling function $\phi$ satisfies \textit{the Strang–Fix condition} of $n$, i.e.
$\h\phi(\xi + k) = O(|\xi|^n)$ as $\xi \to 0$ for $k\in\Z^d$, $k\neq\nul$.
Below, it is convenient to use the following notation: for $\beta \in \Z_d^+$ denote 
$[\beta] = \beta_1+\dots+\beta_d$.
Since $\phi \in H^s(\R^d)$ has compact support, 
then $(\cdot)^{\beta} \phi$, $\beta\in \Z_d^+$, is also in $H^s(\R^d)$ and has compact support. Due to the equality $\h{ (\cdot)^{\beta} \phi} = \left(\frac{1}{-2\pi i} \right)^{[\beta]} D^{\beta} \h\phi$ and by Remark~\ref{rem1} we have that $[ D^{\beta} \h\phi, D^{\beta} \h\phi]_s \in L_{\infty}(\R^d)$ for $[\beta] = n$.
Using Taylor's formula with Lagrange's remainder around the point $\xi=\nul$, we have
$$
|\h\phi(\xi + k)|^2 = \left| \frac{1}{n!} d^n \h\phi(\theta \xi+k, \xi) \right|^2 \le
C \sum_{[\beta]=n} |D^{\beta} \h\phi(\theta \xi+k)|^2 |\xi|^{2n},
$$
where $\theta \in [0,1]$. Hence, for $\xi\in\T^d$
$$
[\h\phi,\h\phi]_s^{\neq\nul}(\xi)  \le 
C \sum_{[\beta]=n}\sum\limits_{k \ne \nul}(1 + |\xi + k|^{2})^s |D^{\beta} \h\phi(\theta \xi+k)|^2 |\xi|^{2n} \le C' |\xi|^{2n}.
$$

The above consideration allows to state simultaneous approximation properties of 
dual wavelet frames in a pair of dual Sobolev spaces.

\begin{thm}
    Let $s \in \N$. Assume that $\{\psi_{\mu,j,k}\}$, $\{\widetilde\psi_{\mu,j,k}\}$ is a MRA-based wavelet system  in $(H^s(\R^d), H^{-s}(\R^d))$ generated by $\phi, \widetilde\phi$. If 
\begin{equation}
\h {{\w b}_{\mu}}(\xi) = O(|\xi|^{n}) \quad
\text{as } \xi \to 0, \quad \mu=2,\dots,u,
        \label{fVM}
    \end{equation}
for some $n>s$, then the MRA-based wavelet system produces a pair of dual wavelet frame in 
    $(H^s(\R^d), H^{-s}(\R^d))$, which provides simultaneous approximation order $(s,n)$ in 
    $H^n(\R^d)$, i.e. for all $f \in H^n(\R^d)$
        \begin{equation*}
        \sum_{r=0}^{s} \norm{M^{*j}}^{-r}
        \left| f - \sum\limits_{k \in \Z^d} \langle f, \w{\phi}_{0k} \rangle \phi_{0k} - 
    \sum_{j=0}^{+\infty}
        \sum_{\mu = 2}^{u}
        \sum_{k\in\mathbb{Z}^d}
        \langle f, \widetilde\psi_{\mu,j,k}\rangle
        \psi_{\mu,j,k} \right|_{H^{r}} \le
        C \norm{M^{*-j}}^{n} |f|_{H^n},
    \end{equation*}
    \label{fDualWF2}
\end{thm}   

\textbf{Proof.} By the perfect reconstruction property
    \begin{equation*}
        \sum_{k\in\mathbb{Z}^d} 
        \langle f, \widetilde\phi_{jk}\rangle
        \phi_{jk} =
        \sum_{k\in\mathbb{Z}^d} 
        \langle f, \widetilde\phi_{0k}\rangle
        \phi_{0k}
        +
        \sum_{i=0}^{j}
        \sum_{\mu=2}^{u}
        \sum_{k\in\mathbb{Z}^d}
        \langle f, \widetilde\psi_{\mu,i,k}\rangle
        \psi_{\mu,i,k}.
    \end{equation*}
    It remains to apply Theorem~\ref{fQjApprox}, since conditions~(\ref{cond6}) are satisfied. $\qed$
    
The problem of constructing MRA‑based wavelet systems in the multivariate case with a given number of vanishing moments~(\ref{fVM}) has been actively studied in the literature. Several concrete examples of such systems can be found in, e.g.,~\cite{Han} (for $M=2 I_d$),~\cite{KPS},~\cite{Kr2019}.


\begin{thebibliography}{40}

\bibitem{Butzer2006}
Bardaro C., Butzer P.L., Stens R.L., Vinti G.: Approximation error of the Whittaker cardinal series in terms of an averaged modulus of smoothness covering discontinuous signals. J. Math. Anal. Appl. \textbf{316}(1), 269--306 (2006)

\bibitem{Butzer2010}
Bardaro C., Butzer P.L., Stens R.L., Vinti G.: Prediction by samples from the past with error estimates covering discontinuous signals. IEEE Trans. Inf. Theory \textbf{56}(1), 614--633 (2010)

\bibitem{SanAntolin2025}
Boukeffous Ch., San Antolín A.: On simultaneous density order from shift invariant subspaces in Sobolev spaces. J. Approx. Theory \textbf{308}, Art. 106147 (2025)

\bibitem{SanAntolinL2}
Boukeffous C., San Antolín A.: On density order of quasi-projection operators and dual wavelet frames. Banach J. Math. Anal. \textbf{19}, 8 (2025)

\bibitem{SanAntolin2026}
Boukeffous C., San Antolín A.: Simultaneous density order of finitely generated shift-invariant subspaces in Sobolev spaces. Rev. Mat. Complut. (2026)

\bibitem{Butzer2005}
Butzer P.L., Higgins J.R., Stens R.L.: Classical and approximate sampling theorems: studies in the $L_p(\mathbb{R})$ and the uniform norm. J. Approx. Theory \textbf{137}(2), 250--263 (2005)

\bibitem{Caravetta91}
Cavaretta A.S., Dahmen W., Micchelli C.A.: Stationary Subdivision. Mem. Amer. Math. Soc. \textbf{453}, Amer. Math. Soc., Providence (1991)

\bibitem{ChuiJiang}
Chui C., Jiang Q.: Balanced multi-wavelets in $\mathbb{R}^s$. Math. Comput. \textbf{74}, 1323--1344 (2005)

\bibitem{Lebedeva}
Gorshanova A.A., Lebedeva E.A.: On the convergence of expansions in systems of dyadic wavelets. Algebra Anal. \textbf{37}(5), 179--197 (2025) [translated as St. Petersburg Math. J.]

\bibitem{Han}
Han B., Shen Z.: Dual wavelet frames and Riesz bases in Sobolev spaces. Constr. Approx. \textbf{29}(3), 369--406 (2009)

\bibitem{HanSmooth}
Han B.: Computing the smoothness exponent of a symmetric multivariate refinable function. SIAM J. Matrix Anal. Appl. \textbf{24}(3), 693--714 (2002)

\bibitem{Jetter2001}
Jetter K., Plonka G.: A survey on $L_2$-approximation orders from shift-invariant spaces. In: Multivariate Approximation and Applications, pp. 73--111. Cambridge University Press, Cambridge (2001)

\bibitem{Jia}
Jia R.-Q.: Approximation by quasi-projection operators in Besov spaces. J. Approx. Theory \textbf{162}(1), 186--200 (2010)

\bibitem{Jia1998ApproxProp}
Jia R.-Q.: Approximation properties of multivariate wavelets. Math. Comp. \textbf{67}, 647--665 (1998)

\bibitem{KKS}
Kolomoitsev Yu., Krivoshein A., Skopina M.: Differential and falsified sampling expansions. J. Fourier Anal. Appl. \textbf{24}(5), 1276--1305 (2018)

\bibitem{KS}
Kolomoitsev Yu., Skopina M.: Quasi-projection operators in weighted $L_p$ spaces. Appl. Comput. Harmon. Anal. \textbf{52}, 165--197 (2021)

\bibitem{KrS11}
Krivoshein A., Skopina M.: Approximation by frame-like wavelet systems. Appl. Comput. Harmon. Anal. \textbf{31}(2), 410--428 (2011)

\bibitem{KrS17}
Krivoshein A., Skopina M.: Multivariate sampling-type approximation. Anal. Appl. \textbf{15}(4), 521--542 (2017)

\bibitem{KPS}
Krivoshein A., Protasov V., Skopina M.: Multivariate Wavelet Frames. Industrial and Applied Mathematics. Springer, Singapore (2016)

\bibitem{KMulti}
Krivoshein A.: Approximation by frame-like multiwavelets. Anal. Appl. \textbf{22}(5), 881--911 (2024)

\bibitem{Kr2019}
Krivoshein A.V.: From frame-like wavelets to wavelet frames keeping approximation properties and symmetry. Appl. Math. Comput. \textbf{344--345}, 204--218 (2019)

\bibitem{Kyriazis}
Kyriazis G.: Approximation of distribution spaces by means of kernel operators. J. Fourier Anal. Appl. \textbf{2}(3), 261--286 (1995)

\bibitem{Zhang18}
Li Y.-Z., Zhang J.-P.: Nonhomogeneous dual wavelet frames and mixed oblique extension principles in Sobolev spaces. Appl. Anal. \textbf{97}(7), 1049--1073 (2018)

\bibitem{RS2}
Ron A., Shen Z.: Affine systems in $L_2(\mathbb{R}^d)$ II: dual systems. J. Fourier Anal. Appl. \textbf{3}, 617--637 (1997)

\bibitem{Rudin}
Rudin W.: Principles of Mathematical Analysis, 3rd edn. McGraw-Hill, New York (1976)

\bibitem{Zhao}
Zhao K.: Simultaneous approximation from PSI spaces. J. Approx. Theory \textbf{81}(2), 166--184 (1995)

	\end{thebibliography}
\end{document}